\documentclass[english]{amsart}

\usepackage{esint}
\usepackage[svgnames]{xcolor} 
\usepackage[colorlinks,citecolor=red,pagebackref,hypertexnames=false,breaklinks]{hyperref}
\usepackage{pgf,tikz}
\usepackage{pdfsync}

\usepackage[bottom]{footmisc}
\usepackage{dsfont}
\usepackage{url}
\usepackage[utf8]{inputenc}
\usepackage[T1]{fontenc}
\usepackage{lmodern}
\usepackage{babel}
\usepackage{mathtools}  
\usepackage{amssymb}
\usepackage{lipsum}
\usepackage{mathrsfs}
\usepackage{color}

\usepackage{epsf,graphicx,epsfig,cite,latexsym}
\usepackage[export]{adjustbox}
\usepackage{mathtools}
\usepackage{latexsym, amsmath, amssymb, a4, epsfig}
\usepackage[T1]{fontenc}

\usepackage{listings}

\usepackage{algorithm}
\usepackage{algpseudocode}

 \usepackage{mathrsfs}
\newcommand{\no}[1]{#1}
\renewcommand{\no}[1]{}
\no{\usepackage{times}\usepackage[subscriptcorrection, slantedGreek, nofontinfo]{mtpro}
\renewcommand{\Delta}{\upDelta}}
\usepackage{color}

\usepackage{tikz}
\usetikzlibrary{calc,trees,positioning,arrows,chains,shapes.geometric,decorations.pathreplacing,decorations.pathmorphing,shapes,matrix,shapes.symbols}

\tikzset{
>=stealth',
  punktchain/.style={
    rectangle, 
    rounded corners, 
    draw=black, very thick,
    text width=10em, 
    minimum height=3em, 
    text centered, 
    on chain},
  line/.style={draw, thick, <-},
  element/.style={
    tape,
    top color=white,
    bottom color=blue!50!black!60!,
    minimum width=8em,
    draw=blue!40!black!90, very thick,
    text width=10em, 
    minimum height=3.5em, 
    text centered, 
    on chain},
  every join/.style={->, thick,shorten >=1pt},
  decoration={brace},
  tuborg/.style={decorate},
  tubnode/.style={midway, right=2pt},
}

\numberwithin{algorithm}{section}
\numberwithin{figure}{section}

\newtheorem{lemma}{Lemma}[section]

\newtheorem{proposition}{Proposition}[section]
\newtheorem{theorem}{Theorem}[section]

\def\la{\langle}
\def\ra{\rangle}
\newcommand{\be}{\begin{equation}}
\newcommand{\ee}{\end{equation}}
\newcommand{\ba}{\begin{array}}
\newcommand{\ea}{\end{array}}

\newcommand{\bea}{\begin{eqnarray*}}
\newcommand{\eea}{\end{eqnarray*}}
\newcommand{\bean}{\begin{eqnarray}}
\newcommand{\eean}{\end{eqnarray}}

\def\hat{\widehat}
\def\tilde{\widetilde}

\def\cydot{\leavevmode\raise.4ex\hbox{.}}

\date{\today}

\numberwithin{algorithm}{section}
\numberwithin{figure}{section}

\title[]{Recovering simultaneously a potential and a point source from Cauchy data}

\author{Gang Bao}

\address{Department of
Mathematics, Zhejiang University,  Hangzhou,
Zhejiang,  310027,    China}

\email{baog@zju.edu.cn}

\author{Yuantong  Liu}

\address{Department of
Mathematics, Zhejiang University,  Hangzhou,
Zhejiang,  310027,    China}

\email{ytliu@zju.edu.cn}

\author{Faouzi Triki}

\address{Faouzi Triki,  Laboratoire Jean Kuntzmann,  UMR CNRS 5224, 
Universit\'e  Grenoble-Alpes, 700 Avenue Centrale,
38401 Saint-Martin-d'H\`eres, France}

\email{faouzi.triki@univ-grenoble-alpes.fr}

\thanks{
The work of GB is supported in part by a NSFC Innovative Group Fund (No.11621101).
The work of FT is supported in part by the
 grant ANR-17-CE40-0029 of the French National Research Agency ANR (project MultiOnde).}

\begin{document}

\lstset{language=Matlab,%
    breaklines=true,%
    morekeywords={matlab2tikz},
    keywordstyle=\color{blue},%
    morekeywords=[2]{1}, keywordstyle=[2]{\color{black}},
    identifierstyle=\color{black},%
    stringstyle=\color{mylilas},
    commentstyle=\color{mygreen},%
    showstringspaces=false,
    numbers=left,%
    numberstyle={\tiny \color{black}},
    numbersep=9pt, 
    emph=[1]{for,end,break},emphstyle=[1]\color{red}, 
}

\begin{abstract}
This paper is devoted to the inverse problem of recovering  simultaneously  a potential and a  point source 
  in a Shr\"odinger equation   from the  associated  nonlinear
 Dirichlet to Neumann map. The  uniqueness of the inversion is  proved and  logarithmic  stability estimates
 are derived. It is well known that  the  inverse problem of determining  only the  potential  while
knowing the source, is  ill-posed.  In contrast the problem of identifying a point source when the 
potential is given is well posed.  The  obtained results show that the  nonlinear Dirichlet to Neumann map  
  contains enough   information  to determine  simultaneously  the potential and  the point source. However recovering 
 a point source imbedded  in an unknown  background medium   becomes  an ill-posed inversion. 
 \end{abstract}

\maketitle

\section{Introduction and main results}
In this paper we study the issue of uniqueness and  stability for determining  simultaneously a smooth
potential  and a point source  in a Shr\"odinger equation by boundary measurements.  Motivation for investigating 
 this  inverse problem is provided by medical imaging as well as antenna synthesis \cite{HR, AKK, El, FKM}. From the point
view of mathematical modeling, many works  have considered  the simplification that the background medium in which 
 the source is imbedded is known. Here we   are interested  in recovering  both the source and the background medium 
 from Cauchy data. \\

 Let $\Omega
\subset \mathbb{R}^{d}, \, d\geq 3, $  be a bounded domain, with  $C^\infty$
boundary $\partial \Omega$. We consider the Shr\"odinger  equation
\begin{equation}\label{eq:eq}
\bigl( \Delta + q(x) \bigr) u(x) = a\delta_{z}(x)\quad \mbox{ in }\quad
\Omega,
\end{equation}
where the real-valued function $q(x)$ is the potential, $z \in \Omega$ is the position of the 
point source, and $a\in \mathbb R^d\setminus\{0\}$
its amplitude.
Assume that the
kernel of the operator $\Delta+ q(x)$ acting on $H_0^1(\Omega)$ is
the trivial space. Associated with \eqref{eq:eq}, we define the
nonlinear Dirichlet-to-Neumann map (DtN) $\Phi[q, a, z]:
H^{1/2}(\partial\Omega)\to H^{-1/2}(\partial\Omega)$ by
\[
\Phi (f) =
\frac{\partial u}{\partial \nu} \biggr|_{\partial \Omega},
\]
where $u$ is the solution to (\ref{eq:eq}) with the Dirichlet
condition $u = f$ on $\partial \Omega$, and $\nu$ is the unit outer
normal vector of $\partial \Omega$.\\

The nonlinear  map $f\rightarrow \Phi[q, a, z](f)$ is 
an affine function. It can be decomposed  as $\Phi[q, a, z](f) = \Phi[q,a,z](0) + \Phi[q, 0, 0](f),$ 
where le latter map is the classical linear Dirichlet-to-Neumann map associated to   the 
Shr\"odinger  equation. \\

We further denote by \[ \Vert \Phi \rVert_\star = \sup_{
\| f\|_{ H^{1/2}(\partial\Omega)} \leq 1} 
\lVert \Phi(f) \rVert_{H^{-1/2}(\partial\Omega)},\]
the norm that we will use to evaluate the strength  of the nonlinear (DtN)  map. Due to the affine 
property of the map, the  trivial  map $\Phi=0$  is the unique solution to   the equation $ \Vert \Phi \rVert_\star = 0$.\\

The inverse problem we consider in this paper is {\it  to recover the triplet $(q, a, z)$
from the knowledge of the nonlinear (DtN)  map $\Phi[q, a, z]$}.
It is well known that  the  inverse problem of determining  only the  potential while 
knowing the source, is  ill-posed.  The uniqueness of this inverse
problem is derived in \cite{SU}. Alessandrini proved that the stability
estimate for this problem is of log type \cite{A88}, and Mandache showed
that the log type stability is optimal  for smooth potentials \cite{Ma}. For a given  potential
it is also well known that the identification of a general source function  
from  full Cauchy data is not possible. Indeed the authors in \cite{AMR} showed 
 that many boundary measurements are not sufficient to fully identify a general source. 
 Moreover it turns out that  increasing the number of boundary measurements does not increase 
 the information concerning the source. The identification may be  achieved by considering many
  boundary measurements generated by different frequencies
 \cite{BLT, BLT2}. Unlike general source functions, point sources are singular and has a lower dimensionality. 
 These specificities  enable one to obtain uniqueness in the inverse  source problem with a single
 Cauchy observed data \cite{ABF}. Many H\"older type  stability estimates have been derived  for this
 inverse problem when 
 the background medium
 is known and homogeneous \cite{EE, EN}.  The inverse problem considered in this 
paper is quite new, and only few partial results are available. Recently the authors in 
\cite{RZ}  established a  H\"older stability estimate on the reconstruction of point sources 
with respect to smooth changes of a known  potential. Their results say that  if the potential
 is known up to a small smooth perturbation,  the recovered source is close to the true one with respect to a 
 given  nonconventional distance (not comparable to classical distances in  Sobolev spaces for example). 
  Now we state the main result of the paper.
\begin{theorem}\label{thm}
Assume that $(q_1, a_1, z_1)$ and $(q_2, a_2, z_2)$ are two triplets  with 
associated (DtN) nonlinear maps $\Phi[q_1, a_1, z_1] $ and $\Phi_2[q_2, a_2, z_2] $, respectively. Let $s
> (d/2) + 1$ and $M \geq 1$. Suppose $\lVert q_{j} \rVert_{H^{s}
( \Omega )} \leq M$ $(j = 1, 2)$ and $\mathrm{supp} ( q_{1} - q_{2} )
\subset \Omega$. \\

Then if   $\lVert
  \Phi[q_1, a_1, z_1] - \Phi_2[q_2, a_2, z_2] \rVert_\star < 1 $, the following 
stability estimate 
\begin{equation}\label{iest}
\lVert a_1\delta_{z_1} -  a_1\delta_{z_2}\rVert_{H^{-s} ( \mathbb{R}^{d} )}+
\lVert q_1-q_2 \rVert_{H^{-s} ( \mathbb{R}^{d} )}
\leq 
C \left(-\log \lVert \Phi[q_1, a_1, z_1] - \Phi_2[q_2, a_2, z_2] \rVert_\star
\right)^{- (s -d/2) }
\end{equation}
holds, where $C>0$ depends only on $s, d,  \Omega, $ and $ M$.
\end{theorem}
The stability of reconstructing the potential and the point source is of  logarithmic 
 type. This means that the inversion is ill-posed and small variations in the measured data 
 can lead to large errors in 
 the reconstructions.   \\
 
The proof of Theorem~\ref{thm} is based on  Alessandrini's arguments in 
\cite{A88} and the Complex geometrical optics  (CGO) solutions constructed in \cite{SU}. The main idea is to 
first recover the potential by exploiting the nonlinearity of the DtN map (Proposition \ref{prop}). Then
 the remaining  inverse problem  becomes a linear one, and  we  again use (CGO)
 solutions to construct  a new type of special  solutions of  the equation 
 \eqref{eq:eq}  in order to  determine the position and 
 amplitude of the point source (Lemma~\ref{mainlemma}).  Using the same approach
  the obtained results  can be extended to the  inverse problem of recovering a potential and
  a finite number of point sources.  In the rest of the paper, we  
  introduce the  (CGO) solutions in Section 2, and  prove the main result in Section 3. 

\section{Complex geometrical optics solutions}
In this section, we construct (CGO) solutions to the equation
(\ref{eq:eq}) by using the idea in \cite{SU}.  We recall the following fundamental results due to Sylvester and Uhlmann in \cite{SU}
concerning solutions of the equation 

\begin{equation}\label{fund-equation}
\Delta w + \xi \cdot \nabla w = f 
\end{equation}
where $\xi \in \mathbb C^{d}$ and $\xi \cdot \xi = 0$. Given $\xi \in \mathbb C^{d}$ with $|\xi| \geq  2$ and $\xi\cdot \xi = 0$, define 
\begin{equation*}
\widehat{K_{\xi}}(k) = {1\over  -|k|^2 + i \xi \cdot k} \quad \mbox{ for } k \in \mathbb R^{d},
\end{equation*}
where $\widehat K_\xi$ stands for the Fourier transform of  the kernel $K_\xi$.
Then for $f \in H^{s}(\mathbb R^{d}),\, s\geq 0 $ with a compact support, $\mathcal K_{\xi}(f):= K_{\xi} * f$ is a solution to the equation 
\begin{equation*}\label{toto}
\Delta w + \xi \cdot \nabla w  = f \mbox{ in } \mathbb R^{d},  
\end{equation*}
and 
\begin{equation*}
\widehat{ K_{\xi}*f}  = \widehat K_{\xi} \cdot \hat f \in H^s(\mathbb R^d). 
\end{equation*}
\begin{lemma}[\cite{SU}] \label{lemma:SUprop21alpha} Let $-1 < \delta < 0$, 
$\xi \in \mathbb C^{d}$ with $|\xi|> 2$ and $\xi \cdot \xi = 0$,  and let $f \in L^{2}_{loc}(\mathbb R^{d})$.  Then 
\begin{equation}\label{SU1-o}
\| K_{\xi} * f \|_{H^s_{\delta}} \leq {C \over |\xi|} \| f \|_{H^s_{\delta+1}} \quad \mbox{ for } s \ge 0, 
\end{equation}
\begin{equation}\label{SU2-o}
\| K_{\xi} * f \|_{H^{s+1}_{\delta}} \leq {C} \| f \|_{H^{s}_{\delta+1}}  \quad \mbox{ for } s \ge 0. 
\end{equation}
for some positive constant  $C>0$ that only depends on $\delta, s , $ and $d$. 
\noindent Here 
\begin{equation*}
\| v\|_{L^{2}_{\delta}} := \| (1 + |\cdot |^{2})^{\delta} v(\cdot)\|_{L^{2}(\mathbb R^d)}
\end{equation*}
and 
\begin{equation*}
\| v\|_{H^{s}_{\delta}} := \sum_{|\alpha| = 0}^{s}\| 
(1 + |\cdot |^{2})^{\delta} \partial^{\alpha } v(\cdot)\|_{L^{2}(\mathbb R^d)}.
\end{equation*}
\end{lemma}

These estimates are the corner stone   of  the proof of the uniqueness of smooth potentials \cite{SU} and of the proof of the stability 
estimates  in \cite{A88}.  By using this lemma, we can obtain a solution to the general  equation
\bean \label{cgoequation}
\Delta \psi + \xi \cdot \nabla \psi + q \psi = f
\eean
satisfying some decaying property as in the following lemma.
\begin{lemma}
\label{lemma:SUtheorem23alpha}
Let $s > (d/2)+1$ be an integer.
Let $\xi \in \mathbb{C}^{d}$ satisfy
$\xi \cdot \xi = 0$ and $\lvert \xi \rvert \geq 2$.
Let $f \in H^{s} ( \Omega )$.
Then there exists constants $C_{1} > 0$ and $C_{2} > 0$ depending only on
$d, s,$ and $\Omega$
such that if
\[
\lvert \xi \rvert
\geq C_{1} \lVert q \rVert_{H^{s} ( \Omega )},
\]
then there exists a solution $\psi \in H^{s} ( \Omega )$
to the equation
{\rm (\ref{cgoequation})}
satisfying the estimates
\bean \label{est:first}
\lVert \psi \rVert_{H^{s} ( \Omega )}
\leq \frac{C_2}{\lvert \xi \rvert }
\lVert f \rVert_{H^{s} ( \Omega )},\\  \label{est:second}
\lVert \psi \rVert_{H^{s+1} ( \Omega )}
\leq C_{2}
\lVert f \rVert_{H^{s} ( \Omega )}.
\eean
\end{lemma}
\proof  
Let $\chi  $ be a $C^\infty$ compactly supported function satisfying $\chi=1$ on a neighborhood 
of $\overline \Omega$. Denote $q_0$ and $f_0$ respectively the extensions in $H^s(\mathbb R^d)$
of the functions $q$ and $f$, satisfying $ \lVert q \rVert_{H^{s} ( \mathbb R^d )} \leq \lVert q \rVert_{H^{s} ( \Omega )},
\, $  $ \lVert f \rVert_{H^{s} ( \mathbb R^d )} \leq \lVert f \rVert_{H^{s} ( \Omega )}$ \cite{McL}, and denote 
 $\tilde q = \chi q_0$ and $\tilde f = \chi f_0$.\\
 
Simple calculation shows that 
\bean \label{ineq1}
  \lVert q \rVert_{H^{s} ( \mathbb R^d )} \leq C\lVert q \rVert_{H^{s} ( \Omega )},\;\; 
   \lVert f \rVert_{H^{s} ( \mathbb R^d )} \leq C\lVert f \rVert_{H^{s} ( \Omega )},
\eean
where $C>0$ is a constant than only depends on $s$ and $d$.\\

 Let $\tilde \psi$
be a solution to the equation \eqref{cgoequation} in the whole space with $\tilde f$ and $\tilde q$ substituting 
 $f$ and $q$ respectively.
Lemma \ref{lemma:SUprop21alpha} shows that the linear  operator $\mathcal K_\xi(\tilde q\cdot)$ is bounded from 
$H^s(\mathbb R^d)$ to itself. Let $I_d$ be the identity operator from $H^s(\mathbb R^d)$ to itself. 
The decaying behavior \eqref{SU2-o} implies that the operator $\mathcal K_\xi(\tilde q\cdot)$
is a contraction for $|\xi|$ large enough, and hence 
 $I_d - \mathcal K_\xi(\tilde q\cdot )$  becomes  invertible. \\
 The estimates \eqref{est:first}  and  \eqref{est:second} follow
  immediately by taking 
$\psi$ the restriction of $\tilde \psi$ to the domain $\Omega$, from the 
convergence of the Neumann series 
 \[\tilde \psi= -\left(I_d - \mathcal K_\xi( \tilde q\cdot )\right)^{-1} \mathcal K_\xi( \tilde q)
=  -\sum_{p=0}^\infty \left(\mathcal K_\xi( \tilde q\cdot )\right)^{p+1},\]
for large $|\xi|$, and  inequalities \eqref{ineq1}. 
\endproof
The needed CGO solutions are constructed as follows.
\begin{proposition}\label{prop:CGO}
Let $s > d/2$ be an integer.
Let $\xi \in \mathbb{C}^{d}$ satisfy
$\xi \cdot \xi = 0$ and $\lvert \xi \rvert \geq 2$.
Define the constants $C_{1}$ and $C_{2}$ as in
Lemma~\ref{lemma:SUtheorem23alpha}.
Then if
\[
\lvert \xi \rvert \geq C_{1} \lVert q \rVert_{H^{s} ( \Omega
)}
\]
then there exists a solution $u$ to the equation {\rm (\ref{eq:eq})}
with the form of
\begin{equation}\label{eq:CGOform}
u(x) = \exp \left( \frac{\xi}{2} \cdot x \right)
\bigl(1 + \psi (x) \bigr) ,
\end{equation}
where $\psi$ has the estimates
\bea
\lVert \psi \rVert_{H^{s} ( \Omega )} \leq \frac{C_{2}}{\lvert \xi \rvert} \lVert q \rVert_{H^{s} ( \Omega )},\\
\lVert \psi \rVert_{H^{s+1} ( \Omega )} \leq C_{2} \lVert q \rVert_{H^{s} ( \Omega )}.
\eea
\end{proposition}
\begin{proof}
Substituting (\ref{eq:CGOform}) into (\ref{eq:eq}),
we have
\[
\Delta \psi + \xi \cdot \nabla \psi + q \psi = - q.
\]
Then by Lemma~\ref{lemma:SUtheorem23alpha},
we obtain this proposition.
\end{proof}

\section{Proof of the stability estimate}

This section is devoted to the proof of Theorem~\ref{thm}. For $f \in H^{1/2}(\partial \Omega)$,  
let $v$ be a solution to
 \begin{equation}\label{eq:eq2}
\bigl( \Delta + q(x) \bigr) v(x) =0 \quad \mbox{ in }\quad
\Omega,
\end{equation}
and satisfying the Dirichlet condition $u=f$ on $\partial \Omega$, and define
 linear map (DtN) $\Phi_0[q]:
H^{1/2}(\partial\Omega)\to H^{-1/2}(\partial\Omega)$ by
\[
\Phi_0(f) =
\frac{\partial v}{\partial \nu} \biggr|_{\partial \Omega}.
\]

We first observe  that  the following inequality 
\bean \label{diff-eq}
\Phi[q, a, z](f) -\Phi[q,a,z](0) =\Phi_0[q](f),      
\eean
 holds for all $f \in H^{1/2}(\partial \Omega)$. 
 \begin{proposition}\label{prop}
Assume that $q_1$ and $q_2$ are two potentials with 
associated (DtN) maps $\Phi_0[q_1] $ and $\Phi_0[q_2] $, respectively. Let $s
> d/2$, $M \geq 1$. Suppose $\lVert q_{j} \rVert_{H^{s}
( \Omega )} \leq M$ $(j = 1, 2)$ and $\mathrm{supp} ( q_{1} - q_{2} )
\subset \Omega$. Then if   $\lVert
  \Phi_0[q_1] - \Phi_0[q_2] \rVert < 1$, the following 
  inequality  
\begin{equation}\label{firstestimate}
\lVert q_1-q_2 \rVert_{H^{-s} ( \mathbb{R}^{d} )}
\leq 
C \left(- \log \left(\lVert \Phi_0[q_1] - \Phi_0[q_2]\rVert_\star
\right) \right)^{-(s-d/2)}
\end{equation}
holds, where $C>0 $ only depends  on $d, s, \Omega, $ and $ M$.

\end{proposition}

Let $v_{j}$ be a solution to {\rm (\ref{eq:eq2})} with $q = q_{j}, \,( j=1, 2),$ 
then we have
\bean \label{prop:identity}
 \int_{\Omega}
 ( q_{2} - q_{1} ) v_{1} v_{2} \,
d x = \bigl\langle
 (\Phi_0[q_1] - \Phi_0[q_2] ) v_{1} |_{\partial \Omega} , \,
 v_{2} |_{\partial \Omega}
\bigr\rangle_{H^{-1/2}, H^{1/2}}.
\eean

Now we would like to estimate $\widehat{q_2-q_1}(\eta), \; \eta \in \mathbb R^d$ in terms of the boundary measurements. 
The principal idea is to estimate  the low frequencies using  products of CGO's solutions, and to approximate  the  high frequencies 
through the regularity of the difference. 
\begin{lemma}\label{lemma:Fourierest}
Let $s > d/2$ be an integer and $M \geq  1$. Assume $\lVert q_{l}
\rVert_{H^{s} ( \Omega )} \leq M$, $\mathrm{supp} ( q_{1} - q_{2}) \subset
\Omega$.  Then there exist  constants $C_M \geq 1 $ such that  the following inequality 
\begin{align}\label{eq:lowerHsestimate}
\lVert q_{2} - q_{1}  \rVert_{H^{-s} ( \mathbb{R}^{d} )}^2 & \leq 
C_M \left( \frac{1}{R^{2s-d}}+ \exp ( C R)  \lVert \Phi_0[q_1] - \Phi_0[q_2]  \rVert^2_\star \right),
\end{align}
holds for all $R>0$, where $C>0$  only depends on 
$s, \Omega $ and $d$. 

\end{lemma}
\proof
In the following proof,  $C$ stands for a general   constant  strictly larger than one
depending only on $d, s$ and $\Omega$. \\

By
Proposition~\ref{prop:CGO}, we can construct CGO solutions $v_{j}
(x)$ to the equation (\ref{eq:eq2}) with $q = q_{j}$,  having the form
of
\[
v_{j} (x) = \exp \left( \frac{\xi_{j}}{2} \cdot x \right)
\bigl( 1 + \psi_{j} (x) \bigr)
\]
for $j = 1, 2$,
and we have
\begin{align}
& \int_{\Omega}
 ( q_{2} - q_{1} )
 \exp \left( \frac{1}{2} ( \xi_{1} + \xi_{2} ) \cdot x \right)
 ( 1 + \psi_{1} + \psi_{2} + \psi_{1} \psi_{2} ) \,
d x \notag \\
& =  \bigl\langle
 (\Phi_0[q_1] - \Phi_0[q_2] ) v_{1} |_{\partial \Omega} , \,
 v_{2} |_{\partial \Omega}
\bigr\rangle_{H^{-1/2}, H^{1/2}}, 
 \label{eq:identityCGO}
\end{align}
from identity~\ref{prop:identity}, where $\psi_{j}$ satisfies
\[
\lVert \psi_{j} \rVert_{H^{s} ( \Omega )} \leq \frac{C}{\lvert
\xi_{j} \rvert} \lVert q_{j} \rVert_{H^{s} ( \Omega )},
\]
if $\xi_{j} \in \mathbb{C}^{d}$ satisfies
$\xi_{j} \cdot \xi_{j} = 0$,
$\lvert \xi_{j} \rvert \geq 2$ and
\begin{equation}\label{eq:xil}
\lvert \xi_{j} \rvert \geq C_{1} \lVert q_{j} \rVert_{H^{s} (
\Omega )} .
\end{equation}

Now, let  $\eta \in \mathbb{R}^{d}$ and $\rho>0$.  We assume that $\alpha, \zeta \in \mathbb{R}^{d}$
satisfy
\begin{equation}\label{eq:kzeta}
\alpha \cdot \eta = \alpha \cdot \zeta = \eta \cdot \zeta = 0,  |\alpha| = \rho, \mbox{
and } \lvert \zeta \rvert^{2} = \lvert \eta \rvert^{2} +  \rho^{2}.
\end{equation}
Define $\xi_{1}$ and $\xi_{2}$ as
\[
\xi_{1} = \zeta + i \alpha - i \eta\quad \mbox{ and }\quad \xi_{2}
= - \zeta - i \alpha - i  \eta .
\]
Then we have
\[
\xi_{j} \cdot \xi_{j} = 0, \
\lvert \xi_{j} \rvert^{2} = \lvert \zeta
\rvert^{2} +  \lvert \eta
\rvert^{2} + \rho^{2} = 2 \lvert \zeta
\rvert^{2} ~ ( l = 1, 2 ) \mbox{ and } \frac{1}{2} ( \xi_{1} +
\xi_{2} ) = - i  \eta .
\]
Hence by (\ref{eq:identityCGO}), we immediately obtain that
\begin{align}
\widehat{q_2-q_1} ( \eta )
& = - \int_{\Omega}
 ( q_{2} - q_{1} ) \exp ( - i  \eta \cdot x )
 ( \psi_{1} + \psi_{2} + \psi_{1} \psi_{2} ) \,
d x \notag \\
& \hspace*{3ex} + 
  \bigl\langle
 (\Phi_0[q_1] - \Phi_0[q_2] ) v_{1} |_{\partial \Omega} , \,
 v_{2} |_{\partial \Omega}
\bigr\rangle_{H^{-1/2}, H^{1/2}},
 \label{eq:Fourierqtilde1}
\end{align}
provided $\lvert \xi_{j} \rvert \geq 2$ and (\ref{eq:xil}) are
satisfied. Suppose now that $\chi \in
C_{0}^{\infty} ( \Omega )$ satisfies $\chi \equiv 1$ near $\Omega$.  
We first estimate the first term on the right hand side
of \eqref{eq:Fourierqtilde1} by
\begin{align*}
& \left\lvert
 \int_{\Omega}
  ( q_{2} - q_{1} ) \exp ( - i  \eta \cdot x )
  ( \psi_{1} + \psi_{2} + \psi_{1} \psi_{2} ) \,
 d x
\right\rvert \\
& = \left\lvert
 \int_{\Omega}
  ( q_{2} - q_{1} ) \exp ( - i  \eta \cdot x )
  \chi ( \psi_{1} + \psi_{2} + \psi_{1} \psi_{2} ) \,
 d x
\right\rvert \\
& \leq \lVert q_{2} - q_{1} \rVert_{H^{-s} ( \Omega )}
\bigl\lVert
 \chi ( \psi_{1} + \psi_{2} + \psi_{1} \psi_{2} )
\bigr\rVert_{H^{s} ( \Omega )}\\
& \leq \lVert q_{2} - q_{1} \rVert_{H^{-s} ( \mathbb{R}^{d} )}
\lVert \chi \rVert_{H^{s} ( \Omega )}
\bigl(
 \lVert \psi_{1} \rVert_{H^{s} ( \Omega )}
 + \lVert \psi_{2} \rVert_{H^{s} ( \Omega )}
 + \lVert \psi_{1} \rVert_{H^{s} ( \Omega )}
 \lVert \psi_{2} \rVert_{H^{s} ( \Omega )}
\bigr) \\
& \leq \frac{C M^2}
         {\lvert \zeta \rvert} \lVert q_{2} - q_{1}  \rVert_{H^{-s} ( \mathbb{R}^{d} )}.
\end{align*}

On the other hand, by  straightforward calculations, we have
\begin{align*}
\bigl\lVert
 v_{j} |_{\partial \Omega}
\bigr\rVert_{L^{2} ( \partial \Omega )}, \bigl\lVert
 \nabla v_{j} |_{\partial \Omega}
\bigr\rVert_{L^{2} ( \partial \Omega )}
& \leq   CM \exp ( C \lvert \zeta \rvert ),
\end{align*}
for $j=1, 2,$ 
which by interpolation \cite{McL}, provide
\[
\bigl\lVert
 v_{l} |_{\partial \Omega}
\bigr\rVert_{H^{1/2} ( \partial \Omega )}
\leq C M \exp ( C \lvert \zeta \rvert ).
\]
Therefore, we can estimate the second term of the right-hand side of
(\ref{eq:Fourierqtilde1}) by
\begin{align*}
 \bigl\langle
 (\Phi_0[q_1] - \Phi_0[q_2] ) v_{1} |_{\partial \Omega} , \,
 v_{2} |_{\partial \Omega}
\bigr\rangle_{H^{-1/2}, H^{1/2}}& \leq\bigl\lVert
 v_{1} |_{\partial \Omega}
\bigr\rVert_{H^{1/2} ( \partial \Omega )}
\bigl\lVert
 v_{2} |_{\partial \Omega}
\bigr\rVert_{H^{1/2} ( \partial \Omega )} \lVert \Phi_0[q_1] - \Phi_0[q_2]  \rVert_\star,
 \\
& \leq C M^2 \exp ( C \lvert \zeta \lvert ) \lVert \Phi_0[q_1] - \Phi_0[q_2]   \rVert_\star.
\end{align*}
Summing up, we have shown that  for $\eta \in
\mathbb{R}^{d}$ if we take $\alpha$
and $\zeta$ satisfying the conditions (\ref{eq:kzeta}), and
\begin{equation}\label{eq:zeta}
\rho \geq C_{1}M+1,
\end{equation}
then
\begin{align}
\lvert \widehat{ q_2-q_1} (\eta ) \rvert & \leq 
\frac{C M^2}
         {\rho} \lVert q_{2} - q_{1}  \rVert_{H^{-s} ( \mathbb{R}^{d} )} +
          C M^2\exp ( C (\rho +  \lvert \eta \lvert ) ) \lVert \Phi_0[q_1] - \Phi_0[q_2]  \rVert_\star,
\label{eq:Fouriersummingup}
\end{align}
holds. Integrating the last inequality with respect to $\eta$ 
over $B_R(0)= \{ \eta \in \mathbb R^d:\,  \lvert \eta \lvert <R \}$, and taking into account the fact that
$s>d/2$, we 
obtain 
\begin{align}
\int_{B_R(0)}\left|\widehat{ q_2-q_1} (\eta ) \right|^2 (1+ \lvert \eta \lvert^2)^{-s} d\eta & \leq 
\frac{C M^4}
         {\rho^2} \lVert q_{2} - q_{1}  \rVert_{H^{-s} ( \mathbb{R}^{d} )}^2 +
          CM^4 \exp ( C (\rho + R) ) \lVert \Phi_0[q_1] - \Phi_0[q_2]  \rVert^2_\star.
\label{eq:Fourierparceval}
\end{align}
Now since $q_2-q_1$ belong to $H^{s} ( \mathbb{R}^{d} )$, we have 
\bean\label{eq:SobolevHsestimate}
\lVert q_{2} - q_{1}  \rVert_{H^{-s} ( \mathbb{R}^{d} )}^2 
 \leq \int_{B_R(0)}\left| \widehat{ q_2-q_1} (\eta ) \right|^2 (1+ \lvert \eta \lvert^2)^{-s} d\eta +\frac{CM^2}{R^{2s-d}}.
\eean
Taking $\rho =  \rho_M $, with $\rho_M^2= 
 (C_1M+1)^2+ 2C^2M^4$, and  combining \eqref{eq:SobolevHsestimate}, and \eqref{eq:Fourierparceval}, we get
\begin{align*}
\lVert q_{2} - q_{1}  \rVert_{H^{-s} ( \mathbb{R}^{d} )}^2 & \leq 
  \frac{CM^2}{R^{2s-d}}+
          C M^4\exp ( C (\rho_M + R) ) \lVert \Phi_0[q_1] - \Phi_0[q_2]  \rVert^2_\star,\\
          &\leq  C_M \left( \frac{1}{R^{2s-d}}+ \exp ( C R)  \lVert \Phi_0[q_1] - \Phi_0[q_2]  \rVert^2_\star \right),
\end{align*}
which finishes the proof of the Lemma.

\endproof
Now, we shall prove Proposition \ref{prop}.

\proof
Since $2s-d>0$,  there  exists a unique $R_0>0$ satisfying 
\bea
\frac{1}{R_0^{2s-d}}= \exp (CR_0)  \lVert \Phi_0[q_1] - \Phi_0[q_2]  \rVert^2_\star,
\eea

where  $C>0$ is the constant  appearing in  Lemma \eqref{lemma:Fourierest}.
Since $\log(R)/R$ is bounded by $e^{-1}$ for all $R>0$, we have 
\bea
R_0 \geq \frac{ -2\log(\lVert \Phi_0[q_1] - \Phi_0[q_2]  \rVert_\star)}{C+\frac{2s-d}{e}}.
\eea
We  deduce from the estimate \eqref{eq:lowerHsestimate} with $R=R_0$, and the previous 
inequality
\begin{align*}
\lVert q_{2} - q_{1}  \rVert_{H^{-s} ( \mathbb{R}^{d} )}^2 & \leq 
  \frac{2C_M}{R_0^{2s-d}}\\
  &\leq 2C_M \left( \frac{-2\log(\lVert \Phi_0[q_1] - \Phi_0[q_2]  \rVert_\star)}{C+\frac{2s-d}{e}}  \right)^{-(2s-d)},
\end{align*}
which achieves the proof of the Proposition.
\endproof

Next, assuming that the potential $q$ is known we 
 identify the source $a\delta_z$ from the knowledge of  
 $\Phi[q, a, z](0)$.
 
 \begin{lemma} \label{mainlemma}
Let $s > (d/2)+1$ be an integer.
Let $\xi \in \mathbb{C}^{d}$ satisfy
$\xi \cdot \xi = 0$,  $\lvert \xi \rvert \geq 2$, and $\xi  \cdot e_1 >0$.
 There 
exist constants $C_{3}>0$ and $C_{4}>0$ that only depend 
on $\Omega$, $s,$ and $d$  such that  if
\[
\lvert \xi \rvert \geq C_{3} \lVert q \rVert_{H^{s} ( \Omega
)},
\]
then there exist solutions $v$ and $w$ to respectively  the equations {\rm (\ref{eq:eq})},
and
\[
\Delta w + \nabla \log(v^2) \cdot \nabla w = 0 \mbox{ in } \Omega,
\]
with the form of
\bean \label{testfunctions}
v(x) \;=\; \exp \left( \frac{\xi}{2} \cdot x \right)
\bigl( 1 + \psi_v (x) \bigr), &
w(x) \;=\; \xi \cdot x +\psi_w(x),
\eean
satisfying  $v \not=0,  \,  \partial_{x_1} w \not=0$ in $\Omega$, 
where $\psi_v$ and  $\psi_w$ have the estimates
\bean \label{estimate6}
\lVert \psi_v \rVert_{H^{s} ( \Omega )} \leq \frac{C_4}{\lvert \xi \rvert}
 \lVert q \rVert_{H^{s} ( \Omega )},&
\lVert \psi_w \rVert_{H^{s} ( \Omega )} \leq C_4 \lVert q \rVert_{H^{s} ( \Omega )}.
\eean
In addition the function $\phi := vw$  lies in $H^{s} ( \Omega )$, and  satisfies the equation {\rm (\ref{eq:eq})}. 
 \end{lemma}
\proof 
Assuming that   $\lvert \xi \rvert \geq C_{1} \lVert q \rVert_{H^{s} ( \Omega
)}$, we deduce from Proposition~\ref{prop:CGO}, the existence of  a solution  $v
\in H^s(\Omega),$ to  the equation {\rm (\ref{eq:eq})} with the form 
\eqref{testfunctions}, and  $\psi_v$, verifying the following estimate
\bean \label{decay2}
\lVert \psi_v \rVert_{H^{s} ( \Omega )} \leq \frac{ C_{2}}{\lvert \xi \rvert} \lVert q \rVert_{H^{s} ( \Omega )},
\eean
where   $C_{1}$ and $C_{2}$ are as in Lemma~\ref{lemma:SUtheorem23alpha}. \\

Since $s > (d/2)+1$, 
$ H^s(\Omega)$ is compactly embedded in $C^1(\overline \Omega)$,  and the inequality
\bean \label{Hsembedding}
\lVert \varphi \rVert_{L^\infty(\Omega )} \leq C_0 \lVert \varphi  \rVert_{H^{s} ( \Omega )},
\eean
is valid for all $\varphi \in H^{s} ( \Omega )$, where $C_0>0$ is a constant that only depends
on $\Omega$, $s,$ and $d$.  \\

Combining the last two inequalities, we get 
\[
\lVert \psi_v \rVert_{L^\infty(\Omega )} \leq \frac{2 C_0 C_{2} }{\lvert \xi \rvert} \lVert q \rVert_{H^{s} ( \Omega )}.
\]
Since $\xi\cdot\xi=0$, taking $\lvert \xi \rvert \geq  \max\{C_{1}, 4C_0 C_{2}  \} \lVert q \rVert_{H^{s} ( \Omega
)}$, leads to $|v| \geq 2\left| \exp \left( \frac{\xi}{2} \cdot x \right)\right| >0$.\\

Since $v \not=0 $ in  $\Omega$, we have $ \nabla  \log(v^2) \in H^{s-1}(\Omega)$ and has the following decomposition
\[
 \nabla  \log(v^2) = \xi+\frac{\nabla \psi_v}{1+\psi_v}.
\] 

Let $\tilde \psi_v \in H^s(\mathbb R^d)$ be a compact  supported extension  of 
$v$ to the whole space  as in the proof of Lemma \ref{lemma:SUtheorem23alpha}, satisfying
\bean \label{ineq2}
  \lVert \tilde \psi_v \rVert_{H^{s} ( \mathbb R^d )} \leq C\lVert \psi_v \rVert_{H^{s} ( \Omega )},
\eean
where $C$ is a constant than only depends on $s$ and $d$.

Denote  now  $ \tilde \psi_w$ the solution to

\[
\Delta \tilde \psi_w+ \xi \cdot \nabla \tilde \psi_w+ \frac{1}{1+\tilde \psi_v} \nabla \tilde \psi_v \cdot \nabla \tilde  \psi_w
 = -  \frac{1}{1+\tilde \psi_v}\nabla \tilde \psi_v \cdot \xi \mbox{ in } \Omega,
\]
Lemma \ref{lemma:SUprop21alpha} shows that the linear  operator 
$\mathcal K_\xi( \frac{1}{1+\tilde \psi_v} \nabla \tilde \psi_v
\cdot)$ is bounded from 
$H^s(\mathbb R^d)$ to itself. Let $I_d$ be the identity operator from $H^s(\mathbb R^d)$ to itself. 
The decaying behavior \eqref{decay2}, and the inequality \eqref{ineq2}
 imply that the operator $K_\xi(\frac{1}{1+\tilde \psi_v} \nabla \tilde \psi_v\cdot \nabla \cdot)$
is a contraction for $|\xi|$ large enough, and hence 
 $I_d - \mathcal K_\xi (\frac{1}{1+\tilde \psi_v} \nabla \tilde \psi_v\cdot  \nabla \cdot )$
becomes  invertible. In fact we have 
\bea
\|\mathcal K_\xi (\frac{1}{1+\tilde \psi_v} \nabla \tilde 
\psi_v\cdot  \nabla \psi) \|_{H^s_{\delta}}
\leq C\|\frac{1}{1+\tilde \psi_v} \nabla \tilde 
\psi_v\cdot  \nabla \psi \|_{H^{s-1}_{\delta+1}}\\
\leq C^\prime \lVert \psi_v \rVert_{H^{s} ( \Omega )} \| \psi\|_{H^s(\mathbb R^d)}\\
\leq  \frac{C^{\prime \prime}}{\lvert \xi \rvert} \lVert q \rVert_{H^{s} ( \Omega )}\| \psi\|_{H^s(\mathbb R^d)},
\eea
for all $\psi \in H^s(\mathbb R^d)$,  and where $C, C^\prime, C^{\prime \prime}>0$ only depends on 
$\Omega$, $s,$ and $d$. \\

  Then 
\bea
\tilde \psi_w= -\left(I_d - \mathcal K_\xi (\frac{1}{1+\tilde \psi_v} 
\nabla \tilde \psi_v\cdot \nabla \cdot)\right)^{-1} \mathcal K_\xi 
(\frac{1}{1+\tilde \psi_v}\nabla \tilde \psi_v \cdot \xi ),
\eea

The  second estimate \eqref{estimate6}  follows immediately by taking 
$\psi_w$ the restriction of $\tilde \psi_w$ to the domain $\Omega$, from the 
convergence of the Neumann series.\\

\endproof

Now we are ready to prove the main  stability estimate. In the following proof,  $C$ stands for a general   constant  strictly larger than one
depending only on $d, s$ and $\Omega$.\\

Let

\bean \label{testfunctions1}
\theta_1(x) = \frac{v(x)}{v(z_1)} \frac{(w(x)-w(z_2))}{w(z_1) -w(z_2)},\\\label{testfunctions2}
 \theta_2(x) = \frac{v(x)}{v(z_2)} \frac{(w(x)-w(z_1))}{w(z_2) -w(z_1)}.
\eean
By construction the functions $\theta_1$ and $\theta_2$ are solutions to the equation \eqref{eq:eq}, and satisfy
$\theta_i(z_j) = \delta_{ij}, \, i, j =1, 2, $ where $\delta_{ij}$  is the Kronecker delta, that is
\bea
\theta_i(z_i)  = 1, & \textrm{and}& \theta_i(z_j)  =0 \;\; \textrm{if}\;\; i\not=j.
\eea

Moreover, we have
\bean
\left \la a_1\delta_{z_1} -a_2\delta_{z_2}, \varphi \right \ra_{H^{-s}, H^s} =  
\left \la a_1\delta_{z_1} -a_2\delta_{z_2}, \varphi(z_1) \theta_1+ \varphi(z_1) \theta_1\right\ra_{H^{-s}, H^s},
\eean
for all $\varphi \in H^s(\Omega)$. Here $ \left \la, \right \ra_{H^{-s}, H^s}$ stands for the dual product between 
$H^{-s}(\Omega)$ and $H^{s}(\Omega)$. \\

Then 
\bea
\|  a_1\delta_{z_1} -a_2\delta_{z_2}\|_{H^{-s}(\Omega)} = \sup_{\varphi \in H^s(\Omega)} 
\left| \left \la a_1\delta_{z_1} -a_2\delta_{z_2}, \varphi(z_1) \theta_1+ \varphi(z_2) \theta_2\right\ra_{H^{-s}, H^s}\right|.
\eea

\begin{proposition} \label{aaa} Let $\theta_i$ be defined  as in \eqref{testfunctions1} and \eqref{testfunctions2}. Then there exists a constant $C>0$ that only 
depends on  $\Omega$, $s,$ and $d$ such that the following inequality 
\bea 
\|  \varphi(z_1) \theta_1+ \varphi(z_2) \theta_2\|_{H^{1/2}(\partial \Omega)} \leq C \|\varphi\|_{H^s(\Omega)},
\eea
is true for all $\varphi \in H^{s}(\Omega)$.
\end{proposition}
\proof
We first prove that there exists a constant $C>0$ that only depends on $\Omega$, $s,$ and $d$ such that
\bean \label{ineq100}
|w(z_2) -w(z_1)| \geq C  |z_2-z_1|.
\eean
Indeed without loss of generality we can choose $z_1$ and $z_2$ on the line $\{te_1; \, t\in \mathbb R \}$,
that is $z_1=  (z_1\cdot e_1) e_1$,  $z_2=  (z_2\cdot e_1) e_1, $ and $ (z_2-z_1)\cdot e_1=  |z_2 -z_1| e_1$. \\

Hence
\bean \label{q1}
w(z_2) -w(z_1)  =  (\xi\cdot e_1)|z_2-z_1| +\psi_w(z_2) -\psi_w(z_1).
\eean
 Since $s> (d/2)+1,$
we have
\bean \label{q2}
|\psi_w(z_2) -\psi_w(z_1)| \leq C \lVert q \rVert_{H^{s} ( \Omega )} |z_2-z_1|.
\eean
Combining equations \eqref{q1} and \eqref{q2}, we get
\bea
|w(z_2) -w(z_1)| \geq  \left|(\xi\cdot e_1) - C \lVert q \rVert_{H^{s} ( \Omega )}\right| |z_2-z_1|.
\eea
Then by choosing $(\xi\cdot e_1)>0$ large enough we obtain \eqref{ineq100}.\\

Back now to the proof of the proposition, we have
\bea
  \varphi(z_1) \theta_1(x)+ \varphi(z_2) \theta_2(x) = (\varphi(z_2) -\varphi(z_1))\theta_2(x)+  \varphi(z_1)( \theta_1(x)+\theta_2(x))\\
  =\frac{v(x)}{v(z_2)} \frac{\varphi(z_2) -\varphi(z_1)}{w(z_2) -w(z_1)}(w(x)-w(z_1)) + \varphi(z_1)\frac{v(x)}{v(z_1)}+
   \varphi(z_1)v(x)\left( \frac{1}{v(x_2)} -\frac{1}{v(x_1)}\right) \frac{1}{w(z_2) -w(z_1)}.
\eea
We finally deduce from  the inequality \eqref{ineq100}, and the regularity of the functions $\varphi$, $v$ and $w$ the desired inequality.
\endproof

For $j=1,2$,  let $u_j$ be the solution to the equation \eqref{eq:eq}  with  zero  Dirichlet boundary condition,  $q_j $, $a_j \delta_{z_j}$,
as a  potential, and  a source respectively. \\

Then  $  U= u_2-u_1$ is a solution to 
\bea \label{eq33}
\bigl( \Delta + q_2(x) \bigr) U(x) = a_2\delta_{z_2}(x) - a_1\delta_{z_1}(x) +u_1(x)(q_1(x)-q_2(x)) \quad \mbox{ in }\quad
\Omega.
\eea

For a given test function $\varphi \in H^s(\Omega)$,
multiplying the previous equation by $\varphi(z_1) \theta_1+ \varphi(z_2) \theta_2$, and integrating by parts, we obtain
\bea
|\left \la a_1\delta_{z_1} -a_2\delta_{z_2}, \varphi \right \ra_{H^{-s}, H^s} | \leq \\
\|\frac{\partial U}{ \partial \nu} \|_{H^{-1/2}(\partial \Omega)} \|  \varphi(z_1) \theta_1+ \varphi(z_2) \theta_2\|_{H^{1/2}(\partial \Omega)}
+\| q_2- q_1\|_{H^{-s}(\mathbb R^d)} \|u_1\|_{H^{s}(\Omega)}\|\varphi\|_{H^{s}(\Omega)}.
\eea
Combining the results of Propositions \ref{aaa} and \ref{prop}, we finish the proof of  the main stability estimate. 


\end{document}